\documentclass[11pt,a4paper]{amsart} 
\usepackage[foot]{amsaddr} 
\usepackage{amsmath} 
\usepackage{amsthm} 
\usepackage{amssymb} 
\usepackage{mathrsfs} 
\usepackage{verbatim} 
\usepackage{aliascnt}
\usepackage{xcolor} \usepackage[colorlinks=true,linkcolor=blue,citecolor=red]{hyperref} 
\usepackage{enumitem} 
\usepackage{multicol} 
\usepackage{graphicx} 
\usepackage{float} 
\usepackage[utf8]{inputenc} 
\usepackage[nameinlink]{cleveref} 
\usepackage{esint} 
\usepackage{ulem}

\newcommand{\barint}{
\rule[.036in]{.12in}{.009in}\kern-.16in \displaystyle\int }

\theoremstyle{plain} 
\newtheorem{theorem}{Theorem} 
\newtheorem*{theorem*}{Theorem} 
\newtheorem{lemma}[theorem]{Lemma} 
 
\newtheorem{corollary}[theorem]{Corollary} 
\newtheorem{remark}[theorem]{Remark}

\theoremstyle{definition} \newtheorem{definition}[theorem]{Definition}      

\newcommand{\Om}{\Omega}

 \newcommand{\diam}{\text{diam}} 
 \numberwithin{equation}{section} \numberwithin{theorem}{section}

\def\diam{\operatorname{diam}}

\def\max{{\rm max\,}}

\def\r{\right}
\def\lf{\left}

\def\eqn#1$$#2$${\begin{equation}\label#1#2\end{equation}} 

\allowdisplaybreaks  

\title[On the measure of the boundary of a Haj\l{}asz--Sobolev $(M^{1,p},M^{1,q})$-extension domain]{On the measure of the boundary of a Haj\l{}asz--Sobolev $(M^{1,p},M^{1,q})$-extension domain} 

\author[R. Alvarado]{$^1$Ryan Alvarado} 
	\email{$^1$rjalvarado@amherst.edu}
	\address{$^1$Department of Mathematics, Amherst College, MA 01002, USA}
\author[R. Mishra]{$^2$Riddhi Mishra}
\email{$^2$riddhi.r.mishra@jyu.fi}
\address{$^2$Department of Mathematics and Statistics, University of Jyv\"askyl\"a,  P.O. Box 35, FI-40014, Jyv\"askyl\"a, Finland}

\thanks{The second named author has been supported by the Academy
of Finland via Centre of Excellence in Analysis and Dynamics Research
(Project number 323960).} 

\dedicatory{Dedicated to Piotr Haj\l{}asz on the occasion of his 60th birthday.}


\begin{document} 
\begin{abstract} 
We prove that the boundary of every bounded $(M^{1,p},M^{1,q})$-extension domain in a $Q$-doubling metric measure space has measure zero whenever $0<q<p<\infty$, with the additional restriction that $p<qQ/(Q-q)$ if $q<Q$.
\end{abstract}

\maketitle

\section{Introduction}
Let $\Omega\subset \mathbb{R}^n$ be a domain (namely, a nonempty, open, and connected set) and let
$1\leq q\leq p<\infty$. 
Recall that $\Omega$ is called a $(W^{1,p},W^{1,q})$-extension domain if every function in the classical Sobolev space $W^{1,p}(\Omega)$ admits an extension belonging to $W^{1,q}(\mathbb{R}^n)$ with norm controlled by the original Sobolev norm. The theory of Sobolev extension domains has been studied extensively and has revealed a close relationship between the existence of extension operators and the geometry of the underlying domain. For example, Haj{\l}asz, Koskela, and Tuominen in \cite{HKT} proved that  every $(W^{1, p}, W^{1, p})$-extension domain
satisfies the following lower $n$-Ahlfors regularity condition:
\begin{equation}\label{eq:ahlfors}
|B(x, r)\cap\Omega|\geq C r^n,
\end{equation}
for every $x\in\overline\Omega$ and  $0<r\leq1$. 
In particular, the boundary of every such domain has  measure zero. The lower $n$-Ahlfors regularity  condition \eqref{eq:ahlfors}, however, is quite restrictive and excludes many domains, including outward cuspidal domains.

The situation is quite different when $q<p$. In particular, the lower measure bound \eqref{eq:ahlfors} need not hold in this case. For instance, it fails at the tips of certain outward cuspidal domains that nevertheless have the $(W^{1,p},W^{1,q})$-extension property for suitable $q<p$; see, for example, \cite{GS:1982,Mazya2,SKV}. Although weaker polynomial lower bounds for the measure of $B(x, r)\cap\Omega$ are known in this case for various ranges of $p$ and $q$, such bounds do not by themselves guarantee that the boundary has measure zero.


The size of the boundary of $(W^{1,p},W^{1,q})$-extension domains has been studied further, for example, in \cite{PUZ,2,1}. In particular, Koskela, Ukhlov, and Zhu \cite{PUZ} constructed a bounded $(W^{1,p},W^{1,q})$-extension domain whose boundary has positive  measure when $1\leq q<n-1$ and $p$ is sufficiently large. However, for  $q>n-1$, they showed that the boundary of every bounded $(W^{1,p},W^{1,q})$-extension domain necessarily has measure zero. More recently, Koskela and Mishra \cite{KM} extended this line of work by proving that every bounded extension domain of this type has a boundary of measure zero whenever $1\leq q<p<\infty$, provided that $p<qn/(n-q)$ when $q<n$.


The purpose of the present paper is to investigate the corresponding problem for the Haj{\l}asz--Sobolev space $M^{1,p}$ in the more general setting of metric measure spaces. Introduced by Haj{\l}asz in \cite{Hajlasz96}, these Sobolev spaces are known to coincide with the classical Sobolev spaces $W^{1,p}$ on $\mathbb{R}^n$ when $1<p<\infty$. Their definition, however, is given by means of a pointwise inequality and therefore remains meaningful in settings where differentiable structures are unavailable. As a result, Haj{\l}asz--Sobolev spaces have played a fundamental role in analysis on metric spaces, particularly in nonlinear potential theory and geometric analysis.  We refer the reader to Section~\ref{preli} for precise definitions.

Let $(X,d,\mu)$ be a metric measure space. Given $0<q\leq p<\infty$, we say that a nonempty measurable set
$\Omega\subset X$ has the $(M^{1,p},M^{1,q})$-\textit{extension property} if every function in $M^{1,p}(\Omega)$ admits an extension belonging to $M^{1,q}(X)$, with norm controlled by the original norm. When $\Omega$ is also open and connected, we will refer to it as an $(M^{1,p},M^{1,q})$-\textit{extension domain}.


For $p=q$, measure bounds and extension problems for Haj{\l}asz--Sobolev spaces are already well understood. Haj{\l}asz, Koskela, and Tuominen \cite[Theorem~5]{HKTR} proved the analogue of \eqref{eq:ahlfors} for $(M^{1,p},M^{1,p})$-extension domains in $Q$-Ahlfors regular geodesic metric measure spaces. This was later extended to doubling metric measure spaces in \cite[Corollary~1.3]{AYY22}, with $r^n$ in \eqref{eq:ahlfors} replaced by the measure of $B(x,r)$.

The case $q<p$, which to the best of our knowledge has not been previously studied for Haj\l{}asz--Sobolev extension domains, even in $\mathbb{R}^n$, is the one we consider here. Since every  $(W^{1,p},W^{1,q})$-extension domain with $1<q\leq p<\infty$ is automatically an $(M^{1,p},M^{1,q})$-extension domain, the result of \cite{KM} provides a natural Euclidean point of comparison. However, in metric measure spaces, one cannot rely on the structure of $\mathbb{R}^n$, and many of the geometric arguments available in the classical setting cease to apply. Therefore, it is not immediate  whether the same boundary-measure conclusion in \cite{KM} should continue to hold for Haj{\l}asz--Sobolev extension domains in the absence of the Euclidean structure.
Our first main result shows that this is indeed the case. 
\begin{theorem}\label{thmE.6}
Let $(X,d,\mu)$ be a $Q$-doubling  metric measure space for some $Q>0$ and suppose that $\Omega \subset X$ is a bounded $(M^{1,p}, M^{1,q})$-extension domain with $0<q<p < q^*_Q$, where $q^*_Q:=\frac{qQ}{Q-q}$ if $q<Q$ and $q^*_Q:=\infty$, otherwise. Then $\mu(\partial \Omega)=0$. 
\end{theorem}

Theorem~\ref{thmE.6} is new even in $\mathbb{R}^n$ equipped with the Lebesgue measure. Moreover, in this setting, one may take $Q=n$ and, for $q>1$, the  boundary-measure result in \cite{KM} follows as a special case of Theorem~\ref{thmE.6}.  

The connectivity assumption in Theorem~\ref{thmE.6} enters the proof in a specific way. For a doubling measure, connectivity of $\Omega$ guarantees a local reverse-doubling condition at every boundary point; see Lemma~\ref{boundary-reverse-doubling}. If such a condition is instead supplied by the ambient measure, connectivity is no longer needed. This leads to our second main result.

\begin{theorem}\label{thmE.6-notconnected}
Let $(X,d,\mu)$ be a metric measure space, where $\mu$ is $Q$-Ahlfors regular for some $Q>0$, and suppose that $\Omega \subset X$ is a bounded nonempty open  set having the $(M^{1,p}, M^{1,q})$-extension property with $0<q<p < q^*_Q$, where $q^*_Q:=\frac{qQ}{Q-q}$ if $q<Q$ and $q^*_Q:=\infty$, otherwise. Then $\mu(\partial \Omega)=0$. 
\end{theorem}


The Ahlfors regularity assumption in Theorem~\ref{thmE.6-notconnected} can be replaced by the weaker assumption that $\mu$ is $Q$-doubling and satisfies an appropriate 
local reverse-doubling condition at almost every boundary point; see Remark~\ref{RDspaces}. 
In particular, the conclusion holds in RD-spaces.

The restriction $p<q_Q^*$ in Theorems~\ref{thmE.6} and \ref{thmE.6-notconnected} is natural and cannot be completely removed. Indeed, the aforementioned example of Koskela, Ukhlov, and Zhu \cite{PUZ}  yields, for appropriate exponents, an $(M^{1,p},M^{1,q})$-extension domain with a boundary of positive measure. Thus,  some upper bound on $p$ remains essential, even in the metric measure context considered here.

The overall strategy for proving Theorem~\ref{thmE.6} is inspired by  \cite{KM} on $(W^{1,p},W^{1,q})$-extension domains; however, its implementation in the metric setting requires a few substantial modifications. Their Euclidean proof relies heavily on properties specific to $\mathbb{R}^n$, including the Ahlfors regularity of Lebesgue measure and the geometric structure of Euclidean balls and annuli. Since these ingredients do not carry over automatically to general metric measure spaces, we instead develop an argument adapted to the metric setting, based on quasiadditive set functions associated with extension operators, local Sobolev-type estimates in doubling spaces, and an iterative measure-decay argument. This approach allows us to obtain the desired boundary measure estimates without appealing to any ambient Euclidean geometry.


The paper is organized as follows. In Section~\ref{preli} we recall the necessary background on Haj{\l}asz--Sobolev spaces and quasiadditive set functions associated with extension operators. Section~\ref{lemmas} contains several auxiliary estimates used in the proofs of the main results. Finally, in Section~\ref{proofs of main results} we combine these ingredients to prove Theorems~\ref{thmE.6} and~\ref{thmE.6-notconnected}.

\subsection{{Notational Conventions}}
Let $\mathbb{N}:=\{1,2,\dots\}$ and $\mathbb{N}_0:=\mathbb{N}\cup\{0\}$. Open (metric) balls in a given metric space $(X,d)$ shall be denoted by $B(x,r)=\{y:\, d(x,y)<r\}$. 
We always denote by $C$ a  positive constant which is independent of the main parameters involved, but it may vary from line to line. We also use $C_{x,r}$ to denote a positive constant depending on the indicated parameters $x$ and $r$.  The characteristic function of a set $E$ will be denoted by $\chi_E$.

\section{Preliminaries}\label{preli}
A triplet $(X,d,\mu)$ is called a \textit{metric measure space} if $(X,d)$ is a metric space and $\mu$ is a nonnegative Borel regular measure on $X$ with the property that the measure of every metric ball is strictly positive and finite. If $\mu$ is \textit{doubling}, that is, there exists a constant $C_{\mathrm{doub}}\geq1$ such that
$$
\mu(B(x,2r))\leq C_{\mathrm{doub}}\mu(B(x,r)),
$$
for every $x\in X$ and $0<r<\infty$, then $(X,d,\mu)$ is called a \textit{doubling metric measure space.}
Given $Q>0$, we say that $\mu$ is  \textit{$Q$-doubling} if there exists a constant $C\geq 1$ such that
$$
\mu(B(y,R))
\leq
C\bigg(\frac{R}{r}\bigg)^Q \mu(B(x,r)),
$$
whenever $x,y\in X$, $0<r\leq R$, and  $B(x,r)\subset B(y,R)$. Moreover, $\mu$ is said to be
\textit{$Q$-Ahlfors regular} if there exists a constant $C\geq 1$ such that
$$
C^{-1}r^Q
\leq
\mu(B(x,r))
\leq
Cr^Q,
$$
for every $x\in X$ and $0<r<\diam(X)$. Note that every $Q$-Ahlfors regular measure is $Q$-doubling, and hence doubling.
Conversely,  every doubling measure is $Q$-doubling for every $Q\geq\log_2(C_{\mathrm{doub}})$.

Let $\mathcal D_{\Omega}(u)$ denote the collection of Haj{\l}asz gradients of $u$ on $\Omega$. That is, $g \in \mathcal D_{\Omega}(u)$
provided $g \colon \Omega \to [0,\infty]$ is measurable and there exists a set
$E \subset \Omega$ with $\mu(E)=0$ such that
\[
    |u(x)-u(y)|
    \le d(x,y)\bigl(g(x)+g(y)\bigr)
\]
for all $x,y \in \Omega \setminus E$. 
If $\Omega=X$ then we will simply write $\mathcal D(u)$ in place of $\mathcal D_{X}(u)$. For $p>0$, the Haj\l{}asz--Sobolev space $M^{1,p}(\Om)$ consists of all $u\in L^{p}(\Om)$ such that $\mathcal{D}_{\Om}(u)\cap L^{p}(\Om) \neq \emptyset$. For $u\in M^{1,p}(\Om)$, we define
\begin{equation}\label{m1pnorm}
    \|u\|_{M^{1,p}(\Om)}:= \|u\|_{L^{p}(\Om)}+ \inf_{g\in \mathcal D_\Omega(u)} \|g\|_{L^{p}(\Om)}.
\end{equation}
When $0<p<1$, the quantity in \eqref{m1pnorm} is only a quasi-norm; for simplicity, we use the term norm throughout.
As is customary, throughout this paper, equalities between measurable functions are understood to hold almost
everywhere.
\begin{definition}
Let $(X,d,\mu)$ be a metric measure space and let $0<q\leq p<\infty$. A nonempty measurable set $\Om\subset X$ is said to have the \textit{$(M^{1,p},M^{1,q})$-extension property} if there exists a map 
    \begin{equation*}
        E: M^{1,p}(\Om) \to M^{1,q}(X),
    \end{equation*}
    and a constant $C>0$ such that, for every $u\in M^{1,p}(\Om)$, $Eu|_{\Om} =u$ and 
    \begin{equation*}
        \|Eu\|_{M^{1,q}(X)}\leq C \|u\|_{M^{1,p}(\Om)}.
    \end{equation*}
\end{definition}

We recall two important estimates from  \cite[Theorem~20  and Remark~21]{agh20}. 
\medskip

\noindent\textit{Local Sobolev inequality:} Let $(X,d,\mu)$ be a $Q$-doubling 
metric measure space and $0< q<Q$.  There exists a positive constant $C$ such that
\begin{equation}\label{GBLS}
\|u\|_{L^{q^*}(B)}\leq 
        \frac{C}{\mu(2B)^{\frac{1}{Q}}}\left[\diam{(B)}\,\|g\|_{L^{q}(2B)}+\|u\|_{L^{q}(2B)}\right]
\end{equation}	
holds for every ball $B\subset X$, $u\in {M}^{1,q}(2B)$, and  $g\in \mathcal D_{2B}(u)$, where $q^*=Qq/(Q-q)$. 
\bigskip


\noindent\textit{Sobolev-Poincar\'e inequality:} Let $(X,d,\mu)$ be a $Q$-doubling 
metric measure space and $0< q<Q$.  There exists a positive constant $C$ such that
\begin{equation}\label{local Sobolev Poincare}
\inf_{\gamma\in\mathbb{R}}\|u-\gamma\|_{L^{q^*}(B)}\leq 
        C\frac{\diam{(B)}}{\mu(2B)^{\frac{1}{Q}}}\,\|g\|_{L^{q}(2B)}
\end{equation}	
holds for every ball $B\subset X$, $u\in {M}^{1,q}(2B)$, and  $g\in \mathcal D_{2B}(u)$, where $q^*=Qq/(Q-q)$.


We need a set function associated with our extension operator. We borrow it from \cite{PUZ}; see  also  \cite{1,2}. For the convenience of the reader, we review its construction and its crucial properties.

\begin{definition}
\label{def: quasiadditive function}
Let $(X,d,\mu)$ be a metric measure space. A set function $\Phi$, defined on the open subsets of $X$ and taking values in $[0,\infty)$, 
    is called \textit{quasiadditive} if it is monotone, that is, 
    \begin{equation*}
        \Phi(U_{1})\leq \Phi(U_{2}),
    \end{equation*}
   whenever $U_{1}, U_{2}\subset X$ are open and $U_{1}\subset U_{2},$ 
    and there exists a positive constant $C$ such that, for every collection $\{U_{i}\}_{i\in \mathbb{N}}$ of pairwise disjoint open subsets of $X$,   \begin{equation*}\label{constant}
        \sum_{i=1}^{\infty}\Phi(U_{i})\leq C\Phi\left(\bigcup_{i=1}^{\infty}U_{i}\right).
    \end{equation*}
\end{definition}
The upper derivative of a quasiadditive set function, $\Phi$, is defined as
\begin{equation*}
    \overline{D\Phi}(x):= \limsup_{r\to0^{+}}{\frac{\Phi(B(x,r))}{\mu(B(x,r))}}.
\end{equation*}
We will use the following variant of a standard
differentiation result for quasiadditive set functions; see \cite{3,RPR}.

\begin{lemma}\label{lemmafun}
Let $(X,d,\mu)$ be a doubling metric measure space, and let $\Phi$ be a quasiadditive set function defined on all open subsets of $X$. Suppose that there exists $M>0$ such that $\Phi(B)\leq M$ for all balls $B\subset X.$ Then, $\overline{D\Phi}(x)<\infty$ for almost every $x\in X$.
\end{lemma}
\begin{proof}
    Fix $x_0\in X$, and for $R>0$, set
    \begin{equation*}
        E_{R}:= \{x \in B(x_0, R): \overline{D\Phi}(x)=\infty\}.
    \end{equation*}
     We will show that $\mu^*(E_{R})=0$, where $\mu^*$ denotes the outer measure associated with $\mu$.\footnote{Passing to $\mu^*$ allows us to avoid any measurability assumption on $D\Phi$.}
     To this end, note that for each $t>0$ and $x\in E_R$, the definition of $\overline{D\Phi}$ yields a radius $r_x<1$ such that 
    \begin{equation*}\label{eq: Phi balls lim sup}
        \Phi(B(x,r_x))> t \mu(B(x,r_x)).
     \end{equation*}
 Consider the family $\mathcal{F}:=\{B(x,r_x): x\in E_{R}\}$ of balls. Appealing to the $5B$-covering lemma (see, e.g, \cite{HKSTbook}), there exists a pairwise disjoint subfamily $\{B_i\}_{i\in I}$ of $\mathcal{F}$ such that
 $$
 E_{R}\subset \bigcup_{B\in\mathcal{F}}B\subset\bigcup_{i\in I} 5B_i.
 $$
 Moreover, since the balls $\{B_i\}_{i\in I}$ are pairwise disjoint, $\mu$ is positive and finite on every ball, and
 $$
 \bigcup_{i\in I}B_i\subset B(x_0, R+1),
 $$
  we can conclude that $I$ is necessarily at most countable.
Therefore,
 \begin{equation}
     \begin{split}
         t \mu^*(E_{R}) &\leq t  \sum_{i=1}^\infty \mu(5B_i)
          \leq Ct \sum_{i=1}^\infty\mu(B_i)\\
         & \leq C \sum_{i=1}^\infty\Phi(B_i)
         \leq  C \Phi\left(\bigcup_{i=1}^\infty B_i\right)\\
         &\leq  C \Phi(B(x_0, R+1))
          \leq CM.
     \end{split}
 \end{equation}
  Hence, $\mu^*(E_{R}) \leq CM/t$, and by
   letting $t\to\infty$, we obtain $\mu^*(E_{R})=0$.
Thus,
$$
\mu^*\bigl(\{x \in X: \overline{D\Phi}(x)=\infty\}\bigr)=\mu^*\left(\bigcup_{n=1}^\infty E_n\right)
\leq\sum_{n=1}^\infty\mu^*(E_n)=0,
$$
from which we conclude that $\overline{D\Phi}(x)<\infty$ for almost every $x\in X$.
\end{proof}
Note that if $\overline{D\Phi}(x)<\infty$, then there exist constants $A_x,r_x>0$,  such that
\begin{equation}\label{funcset}
    \Phi(B(x,r))\leq A_x \mu(B(x,r)),
\end{equation}
for all $0<r<r_x.$

Let $\Omega\subset X$ be a nonempty measurable set having the $(M^{1,p},M^{1,q})$-extension property for some $0<q<p<\infty$.
For an open set $U \subset X$, define $M^{1,p}_0(U,\Omega)$ to be the
collection of all functions $u \in L^p(\Omega)$ such that
$u = 0$ a.e. on $\Omega \setminus U$,
and for which there exists some $g \in \mathcal D_{\Omega}(u) \cap L^p(\Omega)$
satisfying $g = 0$ a.e. on  $\Omega \setminus U$.
We equip $M^{1,p}_0(U,\Omega)$ with the `norm'
\[
    \|u\|_{M^{1,p}_0(U,\Omega)}
    :=
        \left(\int_{U\cap\Omega} |u|^p \, d\mu
        +
        \inf_{g\in\mathcal{D}_\Omega(u)}
    \int_{U\cap\Omega} g^p \, d\mu
    \right)^{1/p},
\]
where the infimum is taken over all functions
$g \in \mathcal D_{\Omega}(u) \cap L^p(\Omega)$ satisfying
$g = 0$ a.e. on $\Omega \setminus U$. 

For every open set $U\subset X$ such that $U\cap\Omega\neq \varnothing$ and each $u\in M_{0}^{1,p}(U,\Omega),$ we define the $q$--Dirichlet energy $\Gamma^{q}_{U}$ on $U$ with respect to the boundary value $u$ by setting
\begin{equation*}\label{sefun}
    \Gamma_{U}^{q}(u):= \inf\left\{\|g\|_{L^q(U)}:v\in M^{1,q}(U),v|_{U\cap \Omega}= u, g\in \mathcal{D}_U(v)\right\}.
\end{equation*}
Let $k>0$ be given by $\frac{1}{k}=\frac{1}{q}-\frac{1}{p}.$ We define a set function $\Phi$ on open subsets $U\subset X$ by 
\begin{equation}\label{setfun}
    \Phi(U):= \sup\left\{\left(\frac{\Gamma_{U}^{q}(u)}{\|u\|_{M^{1,p}_{0}(U,\Omega)}}\right)^k: u\in M_{0}^{1,p}(U,\Omega),\,\,\|u\|_{M^{1,p}_{0}(U,\Omega)}>0\right\}
\end{equation}
whenever $U\cap\Omega\neq\varnothing$, and $\Phi(U):=0$ otherwise. Hereafter, we use the convention $\sup\varnothing:=0$.

The following theorems are slight adaptations of results in \cite{PUZ}. According to the first theorem,  $\Phi$ is a quasiadditive set function. The second theorem gives an important gradient estimate. To prove these theorems, we need the following lemma, which is a small modification of \cite[Lemma~3.2]{PUZ}.

\begin{lemma}\label{Lemma E.4}
 Let $(X,d,\mu)$ be a  metric measure space and suppose that  $\Omega\subset X$ is a nonempty measurable set having the $(M^{1,p},M^{1,q})$-extension property for some $0<q<p<\infty$. Let $U \subset X$ be an open set with $U\cap \Omega \neq \varnothing$. Then for every $u\in M_0^{1,p}(U,\Omega)$ and every $\lambda\geq 0$, we have $\Gamma_U^q(\lambda u)= \lambda\Gamma_U^q(u)$.
\end{lemma}
\begin{proof}
We may assume $\lambda>0$, since otherwise the claim is immediate.
If $v\in M^{1,q}(U)$ is such that $v|_{U\cap \Omega}= u$ and $g\in \mathcal{D}_U(v)$, then $(\lambda v)|_{U\cap \Omega}= \lambda u$ and $\lambda g\in \mathcal{D}_U(\lambda v)$. Hence, $\Gamma_U^q(\lambda u)\leq \lambda\Gamma_U^q(u)$. Repeating this same argument with $\lambda u$ in place of $u$ and $\lambda^{-1}$ in place of $\lambda$ gives the opposite inequality $\Gamma_U^q(u)\leq \lambda^{-1} \Gamma_U^q(\lambda u)$, and the claim follows.
%
\end{proof}

\begin{theorem}\label{thmE.3}
Let $(X,d,\mu)$ be a  metric measure space and suppose that  $\Omega\subset X$ is a nonempty measurable set having the $(M^{1,p},M^{1,q})$-extension property for some $0<q<p<\infty$. Then the set function $\Phi$ defined in \eqref{setfun} is a bounded quasiadditive set function defined on open sets $U\subset X.$
\end{theorem}

\begin{proof} Let $E: M^{1,p}(\Omega)\to M^{1,q}(X)$ be a bounded extension operator, and let $k$ be defined by $\frac{1}{k}=\frac{1}{q}-\frac{1}{p}$. For each open set $U\subset X$ with $U\cap\Omega \neq \varnothing$ and every $u\in M^{1,p}_0(U,\Omega)$, we have
\begin{equation}\label{2.19}
     \Gamma_{U}^{q}(u) \leq \|Eu\|_{M^{1,q}(X)}
     \leq C\|u\|_{M^{1,p}(\Omega)}\leq C \|u\|_{M^{1,p}_0(U,\Omega)}.
\end{equation}
Hence, from \eqref{2.19}, the boundedness of $\Phi$ follows.

Next, let $U_1 \subset U_2 \subset X$ be two open sets. If $U_1 \cap \Omega = \varnothing$, then we have $0 = \Phi(U_1) \leq \Phi(U_2)$. Hence, we assume that $U_1 \cap \Omega \neq \varnothing$. Let $u \in M_0^{1,p}(U_1,\Omega) \subset M_0^{1,p}(U_2,\Omega)$ be arbitrary. Then for each $v \in M^{1,q}(U_2)$ with $v|_{U_2 \cap \Omega} = u$, and $g\in \mathcal{D}_{U_2}(v)$, we have
$\|g\|_{L^{q}(U_2)}\ge\|g\|_{L^{q}(U_1)},$
from which we can deduce that $\Gamma^{q}_{U_2}(u) \ge \Gamma^{q}_{U_1}(u)$. Also, since $u \in M_0^{1,p}(U_1,\Omega) \subset M_0^{1,p}(U_2,\Omega)$ and $U_1 \subset U_2$, we have $\|u\|_{M^{1,p}_{0}(U_1, \Om)}\geq  \|u\|_{M^{1,p}_{0}(U_2,\Om)}$. Hence, we obtain the monotonicity of $\Phi$, that is $\Phi(U_1) \leq \Phi(U_2)$.

To show the quasiadditivity inequality for $\Phi$, let $\{U_i\}_{i=1}^{\infty}$ be a pairwise disjoint collection of open sets. Fix $N \in \mathbb{N}$ and set $U_0 := \bigcup_{i=1}^{N} U_i$. Let $0 < \varepsilon < 1$. For each $i\leq N$, if $\Phi(U_i)=0$,  we set $u_i:=0$. If $\Phi(U_i)>0$, then
using the definition of $\Phi$, we can choose a function $u_i \in M_0^{1,p}(U_i, \Omega)$ such that $\|u_i\|_{M^{1,p}_{0}(U_i,\Omega)}>0$ and
\begin{equation}\label{3.4}
\Gamma^{q}_{U_i}(u_i) \ge \left( \Phi(U_i) \left(1 - \frac{\varepsilon}{2^i}\right) \right)^{\frac{1}{k}}
\|u_i\|_{M^{1,p}_{0}(U_i, \Omega)}.
\end{equation}
Using Lemma~\ref{Lemma E.4} on \eqref{3.4}, we get
$$
\Gamma^{q}_{U_i}(\lambda u_i)=\lambda \Gamma^{q}_{U_i}(u_i) \ge \left( \Phi(U_i) \left(1 - \frac{\varepsilon}{2^i}\right) \right)^{\frac{1}{k}}
\|\lambda u_i\|_{M^{1,p}_{0}(U_i, \Omega)}.
$$
for any $\lambda > 0$. Hence, we may rescale $u_i$ in \eqref{3.4}, so that its norm has any prescribed value. In particular, we can assume
\begin{equation*}\label{3.5}
\|u_i\|_{M^{1,p}_{0}(U_i , \Omega)}^p = \left( \Phi(U_i) \left(1 - \frac{\varepsilon}{2^i}\right)\right).
\end{equation*}
Given this and the choice of $k$, we have
\begin{equation}\label{3.41}
\Gamma^{q}_{U_i}(u_i) \ge \left( \Phi(U_i) \left(1 - \frac{\varepsilon}{2^i}\right) \right)^{\frac{1}{q}}.
\end{equation}

Define $u := \sum_{i=1}^{N} u_i.$
Then $u \in M_0^{1,p}\left( U_0, \Omega \right)$ and, given that the sets $U_i$ are pairwise disjoint, we have
\begin{equation}\label{eq: u norm est}
\|u\|_{M^{1,p}_0(U_0 , \Omega)}\leq\left( \sum_{i=1}^{N}\Phi(U_i) \left(1 - \frac{\varepsilon}{2^i}\right) \right)^{\frac{1}{p}}.
\end{equation}
Let $v \in M^{1,q}\left( U_0\right)$ with $v|_{\Omega\cap U_0} = u$ and let $g\in \mathcal{D}_{U_0}(v)$.
Define $v_i := v|_{U_i}$ for $i = 1, \dots, N$. Then, we have $v_i \in M^{1,q}(U_i)$ with $v_i|_{U_i\cap \Om} = u_i$, $g|_{U_i}\in \mathcal{D}_{U_i}(v_i)$ and, from the definition of $\Gamma^{q}_{U_i}(u_i)$, 
\begin{equation}\label{eqq}
    \|g\|_{L^q(U_i)} \ge \Gamma^{q}_{U_i}(u_i).
\end{equation}
From \eqref{3.41} and \eqref{eqq}, we obtain
\begin{equation}\label{3.7}
\|g\|_{L^q(U_0)} =  \left(\sum_{i=1}^{N}\|g\|_{L^q(U_i)}^q\right)^{\frac{1}{q}}
\ge
\left( \sum_{i=1}^{N}  \Phi(U_i)\left(1 - \frac{\varepsilon}{2^i}\right)
 \right)^{\frac{1}{q}}.
\end{equation}
From \eqref{3.7}, it follows that
\begin{equation}
\label{eq: sneru}
\Gamma^q_{U_0}(u)\ge
\left( \sum_{i=1}^{N}  \Phi(U_i)\left(1 - \frac{\varepsilon}{2^i}\right)
 \right)^{\frac{1}{q}}.
\end{equation}

On the other hand, by \eqref{eq: u norm est} and the definition of $\Phi$, we have
\begin{equation}
\label{eq: sneru-1}
\Gamma^q_{U_0}(u)\leq\Phi(U_0)^{\frac{1}{k}}\|u\|_{M^{1,p}_{0}(U_0 , \Omega)}\leq\Phi(U_0)^{\frac{1}{k}}\left( \sum_{i=1}^{N}\Phi(U_i) \left(1 - \frac{\varepsilon}{2^i}\right) \right)^{\frac{1}{p}}
\end{equation}
Thus, if 
$$
\sum_{i=1}^{N}\Phi(U_i) \left(1 - \frac{\varepsilon}{2^i}\right)>0,
$$
it follows from \eqref{eq: sneru}, \eqref{eq: sneru-1}, and the identity $\frac{1}{k}=\frac{1}{q}-\frac{1}{p}$, that
$$
\sum_{i=1}^{N}\Phi(U_i) \left(1 - \frac{\varepsilon}{2^i}\right)\leq\Phi(U_0).
$$
Clearly this inequality also holds whenever the left-hand-side is zero. Thus, letting $\varepsilon\to0$, we obtain
$$
\sum_{i=1}^{N}\Phi(U_i)\leq\Phi(U_0)=\Phi\left(\bigcup_{i=1}^N U_i\right)\leq\Phi\left(\bigcup_{i=1}^\infty U_i\right).
$$
Since $N$ is arbitrary,
\[
\sum_{i=1}^{\infty} \Phi(U_i)
\le \Phi\!\left( \bigcup_{j=1}^{\infty} U_j \right).
\]
Hence, $\Phi$ is a bounded quasiadditive set function, and the proof of the theorem is complete.
%
%
%
%
%
\end{proof}
The following theorem is immediate from the definition \eqref{setfun} of the set function $\Phi$.
\begin{theorem}\label{thmset}
Let $(X,d,\mu)$ be a  metric measure space and suppose that  $\Omega\subset X$ is a nonempty measurable set having the $(M^{1,p},M^{1,q})$-extension property for some $0<q<p<\infty$. Let $\Phi$ be the set function from \eqref{setfun}. Then, for each ball $B := B(x, r)$ with $x \in \partial \Omega $ and $\Phi(B)>0$ and every function $u \in M^{1,p}_{0}(B,\Omega)$,  there exist a function $v\in M^{1,q}(B)$ and $g\in D_B(v)$ such that $v|_{B\cap \Omega}= u$ and 
  \begin{equation*}
\|g\|_{L^{q}(B)}\leq 2(\Phi(B))^{\frac{1}{k}}\|u\|_{M^{1,p}_{0}(B,\Omega)},
  \end{equation*}
where $\frac{1}{k}=\frac{1}{q}-\frac{1}{p}$. 
\end{theorem}
By combining Theorem~\ref{thmset} and \eqref{funcset} we obtain a key estimate.
\begin{corollary}\label{corollary}
Let $(X,d,\mu)$ be a  metric measure space and suppose that  $\Omega\subset X$ is a nonempty measurable set having the $(M^{1,p},M^{1,q})$-extension property for some $0<q<p<\infty$. Let $\Phi$ be the set function from \eqref{setfun}. Then for every $x\in\partial\Omega$ for which $\overline{D\Phi}(x)<\infty,$ there exist $r_{x}>0$ and $C_{x}>0$ so that, for every $0<r<r_x$ and every $u\in M^{1,p}_{0}(B(x,r),\Omega)$, there exist $v\in M^{1,q}(B(x,r))$ and $g\in D_{B(x,r)}(v)$ such that $v|_{B(x,r)\cap\Omega}=u$ and
\begin{equation*}
\|g\|_{L^{q}(B(x,r))}\leq C_x \mu(B(x,r))^{\frac{1}{k}}\|u\|_{M^{1,p}_{0}(B(x,r),\Omega)},
  \end{equation*}
 where $\frac{1}{k}=\frac{1}{q}-\frac{1}{p}$. 
In particular, if $X$ is doubling then the conclusion holds for almost every $x\in \partial\Omega.$
\end{corollary}
\begin{proof}
Fix $x\in \partial\Omega$ such that $\overline{D\Phi}(x)<\infty$, and let $r_x,A_x>0$ be as in \eqref{funcset}. Fix $0<r<r_x$, set $B:=B(x,r)$, and let $u\in M_0^{1,p}(B,\Omega)$. If $\Phi(B)>0$ then the desired estimate follows immediately from Theorem~\ref{thmset} and \eqref{funcset}.
Now suppose that $\Phi(B)=0$. If $\|u\|_{M_0^{1,p}(B,\Omega)}=0$, then $u=0$ almost everywhere on $\Omega$, and the desired conclusion follows from taking $v=g=0$. If
$\|u\|_{M_0^{1,p}(B,\Omega)}>0$, then from the definition of $\Phi(B)$, we have $\Gamma_B^q(u)=0$.
Therefore, by the definition of $\Gamma_B^q(u)$, for every $\varepsilon>0$,
there exist $v\in M^{1,q}(B)$ and $g\in \mathcal{D}_B(v)$, with
$v|_{B\cap\Omega}=u$, such that
$\|g\|_{L^q(B)}<\varepsilon$. Consequently, since $\mu(B)>0$, the desired estimate follows from choosing
$$
\varepsilon
:= 2A_x^{1/k}\mu(B)^{1/k}\|u\|_{M_0^{1,p}(B,\Omega)}>0.
$$
If $X$ is doubling, then Lemma~\ref{lemmafun} gives $\overline{D\Phi}(x)<\infty$ for
almost every $x\in X$, and hence the conclusion holds for
almost every $x\in\partial\Omega$.
\end{proof}

\section{Lemmas}
\label{lemmas}
\begin{lemma}
\label{small q extension}
Let $(X,d,\mu)$ be a metric measure space and suppose that  $\Omega\subset X$ is a bounded nonempty measurable set having the $(M^{1,p},M^{1,q})$-extension property for some $0<q<p<\infty$. Then $\Omega$ has the $(M^{1,p},M^{1,q_0})$-extension property for all $0<q_0<q$.
\end{lemma}
\begin{proof}
Fix $0<q_0<q$ and let $u\in M^{1,p}(\Omega)$. Then by assumption, there exists $Eu\in M^{1,q}(X)$ such that ${E}u|_{\Omega}=u$ and
$$\|Eu\|_{M^{1,q}(X)}
\leq
C\|u\|_{M^{1,p}(\Omega)},
$$
for some constant $C\in(0,\infty)$ that is independent of $u$.
Since $\Omega$ is bounded, we can choose an $L$-Lipschitz
function $\eta:X\to[0,1]$ such that $\eta$ has bounded support and
$\eta\equiv 1$ pointwise on a neighborhood of $\overline{\Omega}$.
Define $\widetilde{E}u:=\eta Eu$. Then $\widetilde{E}u|_{\Omega}=u$.
Let $g\in L^q(X)$ be a Haj{\l}asz gradient
of $Eu$. By \cite[Lemma~5.20]{hk98}, we have that $\widetilde{E}u\in M^{1,q}(X)$,
where 
$$
G:=\bigl(g+L|Eu|\bigr)\chi_{\operatorname{supp}\eta}\in\mathcal{D}(\widetilde{E}u).
$$
Since $q_0<q$ and $\operatorname{supp}\eta$ has finite measure, by H\"older's inequality,
\[
\bigl\|\widetilde{E}u\bigr\|_{L^{q_0}(X)}
\leq
\mu(\operatorname{supp}\eta)^{\frac{1}{q_0}-\frac{1}{q}}
\|Eu\|_{L^q(X)}<\infty.
\]
Similarly,
\[
\|G\|_{L^{q_0}(X)}
\leq
C
\left(
\|g\|_{L^q(X)}
+
\|Eu\|_{L^q(X)}
\right)<\infty.
\]
Taking the infimum over all Haj{\l}asz gradients $g$ of $Eu$
gives $\widetilde{E}u\in M^{1,q_0}(X)$ with
\[
\bigl\|\widetilde{E}u\bigr\|_{M^{1,q_0}(X)}
\leq
C\|Eu\|_{M^{1,q}(X)}
\leq
C\|u\|_{M^{1,p}(\Omega)}.
\]
Hence, $\Omega$  has the $(M^{1,p},M^{1,q_0})$-extension property, as desired.    
\end{proof}
The following is an abstract iteration scheme established in \cite[Lemma~16]{agh20}, based on an argument used in \cite{korobenkomr}.
\begin{lemma}
\label{iteration}
Let $(X,d,\mu)$ be a metric measure space and let $\Omega\subset X$ be a nonempty open set.
Suppose that $a,b,p,t,\theta\in(0,\infty)$ satisfy $a<b$ and $p<t$.
Let  $x\in \Omega$ and $\{r_j\}_{j\in\mathbb{N}}$  be  a sequence of radii  such that
$a\leq r_j\leq b$
and
$$
\lf[\mu(B(x,r_{j+1})\cap\Omega)\r]^{1/t}\leq\theta 2^j \lf[\mu(B(x,r_j)\cap\Omega)\r]^{1/p},
\
\forall\,j\in\mathbb{N}.
$$
Then
$$\mu(B(x,r_{1})\cap\Omega)\geq \theta^{-pt/(t-p)}\,2^{-pt^2/(t-p)^2}.$$
\end{lemma}
The assumption that $\Omega$ is nonempty and open in Lemma~\ref{iteration} is only used to ensure that $(\Omega,d,\mu)$ is a metric measure space.

Recall that, given $q,Q\in(0,\infty)$, we set $q^*_Q:=\frac{qQ}{Q-q}$ if $q<Q$ and $q^*_Q:=\infty$, otherwise. Here and thereafter, we define $\frac{1}{\infty}:=0$.

\begin{lemma}
\label{lem: lower measure bound}
Let $(X,d,\mu)$ be a metric measure space, where $\mu$ is a $Q$-doubling measure for some $Q\in(0,\infty)$. Suppose $\Omega\subset X$ is a bounded nonempty open set having the $(M^{1,p},M^{1,q})$-extension property for some $0< q <p<q^*_Q$, and let $s_0:=\frac{1}{\frac{1}{p}-\frac{1}{q^*_Q}}\geq\max\{p,Q\}$.
Then, for every $s\in[s_0,\infty)$, there exists $\theta\in(0,\infty)$ such that
\begin{equation}
\label{Lbound}
\mu(B(x,r)\cap\Omega)\geq\theta r^s,
\end{equation}
for every $x\in\overline{\Omega}$ and all $r\in(0,{\rm diam}(\Omega)]$, where the value $s=s_0$ is only permissible when $q<Q$.
\end{lemma}
\begin{proof}
We first establish \eqref{Lbound} for points $x\in\Omega$. To this end, fix  $x\in\Omega$ and $r\in(0,{\rm diam}(\Omega)]$. 
We define a collection of functions $\{u_j\}_{j\in\mathbb{N}}$
as follows: for each $j\in\mathbb{N}$, let 
$r_j:=(2^{-j-1}+2^{-1})r$ and observe that, for all $j\in\mathbb{N}$,
\begin{equation}
\label{radii}
\frac{1}{2}r<r_{j+1}<r_j\leq\frac{3}{4}r.
\end{equation}
Then, for each $j\in\mathbb{N}$,  define $u_j:\Omega\to[0,1]$ by setting for each $y\in \Omega$,
\begin{eqnarray*}
u_j(y):=
\left\{
\begin{array}{ll}
\,\qquad 1\quad &\mbox{if $y\in{B}(x,r_{j+1})\cap\Omega$,}
\\[6pt]
\displaystyle\frac{r_{j}-d(x,y)}{r_{j}-r_{j+1}}
&\mbox{if $y\in [B(x,r_{j})\cap\Omega]\setminus B(x,r_{j+1})$,}
\\[6pt]
\,\qquad 0 &\mbox{if $y\in \Omega\setminus B(x,r_{j})$.}
\end{array}
\right.
\end{eqnarray*}
It is straightforward to check that $u_j$ is a $(r_j-r_{j+1})^{-1}$-Lipschitz function on $\Omega$ supported in $B(x,r_{j})\cap\Omega$, where $\displaystyle (r_j-r_{j+1})^{-1}=2^{j+2}r^{-1}$, and that $$g_j:=2^{j+2}r^{-1}\chi_{B(x,r_{j})\cap\Omega}\in \mathcal{D}_\Omega (u_j).$$
In particular, we have that $u_j\in M^{1,p}(\Omega)$ with
\begin{equation*}\label{xu-12-X}
\|g_j\|_{L^{p}(\Omega)}\leq 
r^{-1}\,2^{j+2}\lf[\mu(B(x,r_j)\cap\Omega)\r]^{1/p}
\end{equation*}
and
\begin{equation*}\label{xu-13-X}
\Vert u_j\Vert_{L^{p}(\Omega)}\leq \lf[\mu(B(x,r_j)\cap\Omega)\r]^{1/p}.
\end{equation*}
From this, and the fact that $1\leq r^{-1}{\rm diam}(\Omega)$, we have
\begin{align}
\label{djq-476}
\|u_j\|_{{M}^{1,p}(\Omega)}&\leq\|g_j\|_{L^{p}(\Omega)}+\Vert u_j\Vert_{L^{p}(\Omega)}\nonumber\\
&\leq r^{-1}(2^{j+2}+{\rm diam}(\Omega))\,\lf[\mu(B(x,r_j)\cap\Omega)\r]^{1/p}\nonumber
\\
&\leq cr^{-1}2^{j+3}\,\lf[\mu(B(x,r_j)\cap\Omega)\r]^{1/p},
\end{align}
where $c:=\max\{1,{\rm diam}(\Omega)\}$.
Moreover, since $u_j\in M^{1,p}(\Omega)$, we have that $Eu_j\in M^{1,q}(X)$,
with
\begin{equation}
\label{embed}
\|Eu_j\|_{{M}^{1,q}(X)}\leq C\|u_j\|_{{M}^{1,p}(\Omega)}.
\end{equation}
To proceed, we assume for the moment that $q<Q$ and fix a ball $B_0:=B(x_0,R_0)$ large enough so that $\overline{\Omega}\subset B_0$. Then, 
by the local Sobolev inequality \eqref{GBLS},
\begin{equation}
\label{embed2}
\|Eu_j\|_{L^{q^*_Q}(B_0)}\leq C\|Eu_j\|_{{M}^{1,q}(2B_0)}
\leq C\|Eu_j\|_{{M}^{1,q}(X)},
\end{equation}
where the constant $C$ depends on $B_0$, which ultimately only depends on $\Omega$.
By combining \eqref{embed} and \eqref{embed2}, we conclude that
\begin{equation}
\label{embed3}
\|Eu_j\|_{L^{q^*_Q}(B_0)}\leq C\|u_j\|_{{M}^{1,p}(\Omega)}.
\end{equation}
Since $u_j\equiv1$ on the set $B(x,r_{j+1})\cap\Omega\subset B_0$, and ${Eu_j}|_\Omega=u_j$, one deduces that
\begin{equation}\label{xu-14-X}
\|Eu_j\|_{L^{q^*_Q}(B_0)}\ge \Vert u_j\Vert_{L^{q^\ast_Q}(B(x,r_{j+1})\cap\Omega)}=\lf[\mu(B(x,r_{j+1})\cap\Omega)\r]^{1/q^*_Q}.
\end{equation}
Combining \eqref{djq-476}, \eqref{embed3}, and \eqref{xu-14-X} gives
\begin{equation*}
\label{xu-15-X}
\lf[\mu(B(x,r_{j+1})\cap\Omega)\r]^{1/q^*_Q}
\leq C r^{-1}2^{j}\lf[\mu(B(x,r_j)\cap\Omega)\r]^{1/p},
\end{equation*}
for all $j\in\mathbb{N}$,
where $C\in(0,\infty)$ is independent of both $x$ and $r$. In view of \eqref{radii},
we can invoke Lemma~\ref{iteration} for the metric measure space $(\Omega,d,\mu)$ with $p:=p$, $t:=q^*_Q$, and $\theta:=C\,r^{-1}$, to conclude that
\begin{equation*}
\label{xu-16-X}
\mu(B(x,r)\cap\Omega)\geq\mu(B(x,r_1)\cap\Omega)\geq C
r^{\frac{1}{\frac{1}{p}-\frac{1}{q^*_Q}}}=Cr^{s_0},
\end{equation*}
where the first inequality follows from \eqref{radii}.
Hence, the lower measure condition \eqref{Lbound} holds with $s=s_0$. 
It is easy to check that this estimate also holds for every $s>s_0$, possibly after adjusting the constant $\theta$.

We now consider the case $q\geq Q$. Fix $s>s_0=p$, and define
\[
\kappa:=\frac{ps}{s-p}
\quad\mbox{and}\quad
q_0:=\frac{Q\kappa}{Q+\kappa}.
\]
Then
\[
0<q_0<Q\leq q<p<\kappa=\frac{q_0Q}{Q-q_0}=
(q_0)_Q^*.
\]
By Lemma~\ref{small q extension}, $\Omega$ has the $(M^{1,p},M^{1,q_0})$-extension property, where $0<q_0<Q$.
Granted this, we may now repeat the proof for the case $q<Q$, with $q_0$ in
place of $q$, and since
\[
\frac{1}{\frac{1}{p}-\frac{1}{(q_0)_Q^*}}
=
s,
\]
we obtain $\mu(B(x,r)\cap\Omega)\geq c_sr^s$
for every $x\in \Omega$ and every $0<r\leq {\rm diam}(\Omega)$.

To finish the proving the lemma,
let $x \in \overline{\Omega}$ and $0<r\leq {\rm diam}(\Omega)$. Choose \(y\in\Omega\) such that $d(x,y)<r/2$. Then $B(y,r/2)\subset B(x,r)$ and thus, by the estimate already proved for centers in $\overline{\Omega}$,
\begin{align*}
\mu(B(x,r)\cap\Omega)
&\geq \mu(B(y,r/2)\cap\Omega)\\
&\geq \theta\left(r/2\right)^s
=2^{-s}\theta r^s.
\end{align*}
Therefore, the estimate \eqref{Lbound} holds for every $x\in\overline{\Omega}$, after replacing $\theta$ by $2^{-s}\theta$.
\end{proof}

For a set $\Omega\subset X$ and a point $x\in\partial\Omega$, we define $A_{r_{1},r_{2},\Omega}$ by
\begin{equation*}
    A_{r_{1},r_{2},\Omega}:=\{y\in\Omega: r_{1}\leq d(x,y)<r_{2}\}=[B(x,r_2)\setminus B(x,r_1)]\cap\Omega,
\end{equation*}
where $r_{2}>r_{1}>0.$ Note that
\begin{equation}\label{ball-annulus-decomposition}
    B(x,r_2)\cap\Omega
    =
    \bigl(B(x,r_1)\cap\Omega\bigr)
    \cup
    A_{r_1,r_2,\Omega},
\end{equation}
where the union is disjoint.

\begin{lemma}\label{lemma1}
Let $(X,d,\mu)$ be a metric measure space, where $\mu$ is a $Q$-doubling measure for some $Q\in(0,\infty)$.  Suppose that  $\Omega\subset X$ is a bounded nonempty measurable set having the $(M^{1,p},M^{1,q})$-extension property for some $0<q<Q$ and $q<p<q^*_Q$, and let $\Phi$ be the set function from \eqref{setfun}. Fix $x\in \partial \Omega,$ and suppose that
   \begin{equation*}
   \label{eq: lim ratio to zero}
       \limsup_{r\to 0}\frac{\mu(B(x,r)\cap \Omega)}{\mu(B(x,r))}=0\quad\mbox{and}\quad\overline{D\Phi}(x)<\infty.
   \end{equation*}
  Then, for every $0<c<1<L$, there exists $r_{x,c,L}>0$ such that
   \begin{equation*}
       \min\{\mu(A_{r,Lr,\Omega}),\mu(B(x,r/2)\cap\Omega)\}\leq c\mu(B(x,r)\cap\Omega),
   \end{equation*}
   whenever $0<r<r_{x,c,L}.$
\end{lemma}
\begin{proof}
Suppose $x\in \partial \Omega$ is as in the statement of the present lemma and fix $0<c<1<L$.
Given any $\delta>0$, our assumptions give the existence of $0< r_{x,\delta}<1$ such that
\begin{equation}\label{besti}
    \mu(B(x,r)\cap \Omega)\leq \delta\mu(B(x,r)),
\end{equation}
whenever $0<r<r_{x,\delta}.$

Consider  the function $u:\Omega\to[0,1]$ defined by setting for each $y\in\Omega$,
\begin{eqnarray*}
u(y):=
\left\{
\begin{array}{ll}
\,\qquad 1\quad &\mbox{if $y\in B(x,r/2)\cap\Omega$,}
\\[6pt]
\displaystyle \frac{2r-2d(x,y)}{r}
&\mbox{if $y\in [B(x,r)\cap\Omega]\setminus B(x,r/2)$,}
\\[6pt]
\,\qquad 0 &\mbox{if $y\in \Omega\setminus B(x,r)$.}
\end{array}
\right.
\end{eqnarray*}
Then $u$ is $\frac{2}{r}$-Lipschitz and hence the function
\begin{equation*}
 g_u:=  \frac{2}{r} \chi_{B(x,r)\cap \Om}
\end{equation*}
is a Haj{\l}asz gradient of $u$. In particular, $u\in M^{1,p}_{0}(B(x,2Lr),\Omega)$, and when $r<1$,
\begin{align}
\label{eq: u estimate}
\|u\|_{M^{1,p}_{0}(B(x,2Lr),\Omega)}&\leq C
\|u\|_{L^{p}(B(x,r)\cap\Omega)}
+\|g_u\|_{L^{p}(B(x,r)\cap\Omega)}\notag\\
&\leq C\mu(B(x,r)\cap\Omega)^{\frac{1}{p}}+Cr^{-1}\mu(B(x,r)\cap\Omega)^{\frac{1}{p}}\notag\\
&\leq Cr^{-1}\mu(B(x,r)\cap\Omega)^{\frac{1}{p}}.
\end{align}
Since $\overline{D\Phi}(x)<\infty,$ Corollary~\ref{corollary} implies that there exist $r_{x}>0$ and $C_x>0$ such that, for every $0<2Lr<r_{x},$ there exist $v\in M^{1,q}(B(x,2Lr))$ and $g_{v}\in D_{B(x,2Lr)}(v)$ such that $v|_{B(x,2Lr)\cap\Omega}= u$ and
\begin{equation}\label{puz}
    \|g_{v}\|_{L^q(B(x,2Lr))}\leq C_{x}\mu(B(x,2Lr))^{\frac{1}{q} -\frac{1}{p}}\|u\|_{M^{1,p}_0(B(x,2Lr),\Omega)}.
\end{equation}
By the definition of $u$ we have that $v=1$ on $B(x,r/2)\cap \Omega$ and $v=0$ on $A_{r,Lr,\Omega}$. Hence, for every $\gamma\in\mathbb{R}$, we have
$$
\|v-\gamma\|_{L^{q^*_Q}(B(x,Lr))}\geq
\max\left\{|\gamma|\,\mu(A_{r,Lr,\Omega})^{\frac{1}{q^*_Q}}, |1-\gamma|\,\mu(B(x,r/2)\cap \Omega)^{\frac{1}{q^*_Q}}\right\}.
$$
Since $\max\{|\gamma|,|1-\gamma|\}\geq\frac{1}{2}$, we have
\begin{equation}
\label{eq: v est B2Lr}    
\inf_{\gamma\in\mathbb{R}}\|v-\gamma\|_{L^{q^*_Q}(B(x,Lr))}
\geq\frac{1}{2}\min\left\{\mu(A_{r,Lr,\Omega})^{\frac{1}{q^*_Q}},\mu(B(x,r/2)\cap \Omega)^{\frac{1}{q^*_Q}}\right\}.
\end{equation}
Using the Sobolev-Poincar\'e inequality \eqref{local Sobolev Poincare}, \eqref{eq: v est B2Lr}, \eqref{puz},
\eqref{eq: u estimate}, \eqref{besti}, and the fact that $p<q^*_Q$, we deduce that
\begin{align*}
\min\left\{\mu(A_{r,Lr,\Omega})^{\frac{1}{q^*_Q}},\mu(B(x,r/2)\cap \Omega)^{\frac{1}{q^*_Q}}\right\}&
\leq C\frac{r}{\mu(B(x,2Lr))^\frac{1}{Q}}\,\|g_{v}\|_{L^q(B(x,2Lr))}\\
&\leq Cr\mu(B(x,2Lr))^{\frac{1}{q_Q^*}-\frac{1}{p}}\,\|u\|_{M^{1,p}_{0}(B(x,2Lr),\Omega)}\\
&\leq C\mu(B(x,2Lr))^{\frac{1}{q_Q^*}-\frac{1}{p}}
\mu(B(x,r)\cap\Omega)^{\frac{1}{p}}\\
&\leq C\mu(B(x,r))^{\frac{1}{q_Q^*}-\frac{1}{p}}
\mu(B(x,r)\cap\Omega)^{\frac{1}{p}}\\
&=C\left(\frac{\mu(B(x,r)\cap\Omega)}{\mu(B(x,r))}\right)^{\frac{1}{p}-\frac{1}{q_Q^*}}\mu(B(x,r)\cap\Omega)^\frac{1}{q_Q^*}
\\
&\leq C\delta^{\frac{1}{p}-\frac{1}{q_Q^*}}\mu(B(x,r)\cap\Omega)^\frac{1}{q_Q^*}
\end{align*}
Hence,
$$
\min\{\mu(A_{r,Lr,\Omega}),\mu(B(x,r/2)\cap \Omega)\}\leq C\delta^{\frac{q^*_Q}{p}-1}\mu(B(x,r)\cap\Omega).
$$
Since $\frac{q^*_Q}{p}-1>0$, we can choose $\delta>0$ such that $\delta^{\frac{q^*_Q}{p}-1}\leq\frac{c}{C},$ and the desired estimate follows.
\end{proof}


To obtain the upper volume bound, we employ the following lemma.
\begin{lemma}\label{upper}
Let $(X,d,\mu)$ be a metric measure space, where $\mu$ is a $Q$-doubling measure for some $Q\in(0,\infty)$. 
Suppose that  $\Omega\subset X$ is a bounded nonempty measurable set having the $(M^{1,p},M^{1,q})$-extension property for some $0<q<Q$ and $q<p<q^*_Q$, and let $\Phi$ be the set function from \eqref{setfun}. Fix $x\in \partial \Omega,$ and suppose that
\begin{equation*}
\limsup_{r\to 0}\frac{\mu(B(x,r)\cap \Omega)}{\mu(B(x,r))}=0\quad\mbox{and}\quad\overline{D\Phi}(x)<\infty.
\end{equation*}
Assume further that $\mu$ satisfies the following local reverse-doubling condition  at $x$:\ there exist constants $C_{\mathrm{rd}}\geq1$, $\alpha>0$, and $r_{\mathrm{rd}}>0$
such that
\begin{equation}\label{local-reverse-doubling at x}
    \mu(B(x,r))
    \leq
    C_{\mathrm{rd}}
    \left(\frac{r}{R}\right)^\alpha
    \mu(B(x,R))
\end{equation}
whenever $0<r\leq R<r_{\mathrm{rd}}.$
Then, for every $0<c<1$, there exists $r_{x,c}>0$ such that
\begin{equation}
\label{upper-decay-conclusion}
\mu(B(x,r/2)\cap\Omega)\leq c\mu(B(x,r)\cap\Omega),
\end{equation}
whenever $0<r<r_{x,c}.$
\end{lemma}
\begin{proof}
First, choose a constant $\tau>0$ sufficiently small so that
$\tau<c$, $\tau(1+\tau)<1$, and $1+\tau<2^\alpha$.
By Lemma~\ref{lemma1} with $L=4$ and $\tau$ in place of $c$, there exists $\widetilde{r}_{x,\tau}>0$ such that
\begin{equation}\label{uslemma}
\min\{\mu(A_{t,4t,\Omega}),\mu(B(x,t/2)\cap\Omega)\}\leq \tau\mu(B(x,t)\cap\Omega),
\end{equation}
whenever $0<t<\widetilde{r}_{x,\tau}.$

Observe that
\begin{equation}
\label{eq: singleton measure zero}
\mu(\{x\}\cap\Omega)=0.
\end{equation}
Indeed, by our assumptions, for $0<t<1$, we have
$$
\mu(B(x,t)\cap\Omega)=\frac{\mu(B(x,t)\cap\Omega)}{\mu(B(x,t))}\mu(B(x,t))\leq\frac{\mu(B(x,t)\cap\Omega)}{\mu(B(x,t))}\mu(B(x,1))\to0
$$
as $t\to 0$. Hence,
$$
\mu(\{x\}\cap\Omega)=\lim_{t\to0}\mu(B(x,t)\cap\Omega)=0,
$$
and \eqref{eq: singleton measure zero} follows.

Now, suppose to the contrary that there exist arbitrarily small radii $r>0$ satisfying 
\begin{equation}
\label{eq: contra}
\mu(B(x,r/2)\cap\Omega)>c\mu(B(x,r)\cap\Omega).
\end{equation}
In particular, observe that we necessarily have that $\mu(B(x,r)\cap\Omega)>0$ for any radius satisfying \eqref{eq: contra}. Choose a radius $0<\rho<\min\{\widetilde{r}_{x,\tau},r_{\mathrm{rd}}\}$ satisfying \eqref{eq: contra}.
Observe that we can write
$$
0<\mu(B(x,\rho)\cap\Omega)=\mu\bigg(\bigl(\{x\}\cap\Omega\bigr)\cup\bigcup_{m=0}^\infty A_{2^{-m-1}\rho,2^{-m}\rho,\Omega}\bigg)
=\mu\bigg(\bigcup_{m=0}^\infty A_{2^{-m-1}\rho,2^{-m}\rho,\Omega}\bigg).
$$
Thus, there exists an integer $m\geq0$ such that
\begin{equation}
\label{eq: choice of sigma}
\sigma:=\mu(A_{2^{-m-1}\rho,2^{-m}\rho,\Omega})>0.
\end{equation}
Set $R_0:=2^{-m-1}\rho$ and note that $R_0<\min\{\widetilde{r}_{x,\tau},r_{\mathrm{rd}}\}$.
Suppose that $0<r<R_0/2$ and that \eqref{eq: contra} holds. Since $\tau<c$,
inequality \eqref{eq: contra} gives
$$
\mu(B(x,r/2)\cap\Omega)>\tau\mu(B(x,r)\cap\Omega),
$$
which together with \eqref{uslemma} and $\mu(B(x,r)\cap\Omega)>0$, implies
that
\begin{equation}\label{eq: annuli estimate 1}
\mu(A_{r,4r,\Omega})=\min\{\mu(A_{r,4r,\Omega}),\mu(B(x,r/2)\cap\Omega)\}\leq\tau\mu(B(x,r)\cap\Omega).
\end{equation}

We now iterate this process. For each $j=0,1,2,\dots$, set $r_j:=2^jr$, and let $N=N(r)$ be the unique nonnegative integer satisfying
\begin{equation}\label{choice-of-N}
\frac{R_0}{2}\leq r_N=2^Nr<R_0.
\end{equation}
We claim that, for every $j=0,\ldots,N$, 
\begin{equation}
\label{eq: induction annuli estimte}
\mu(A_{r_j,4r_j,\Omega})\leq \tau\mu(B(x,r_j)\cap\Omega)
\end{equation}
and
\begin{equation}
\label{eq: induction estimte}
\mu(B(x,r_j)\cap\Omega)
\leq(1+\tau)^j\mu(B(x,r)\cap\Omega).
\end{equation}
Proceeding  inductively, observe that for $j=0$, \eqref{eq: induction estimte} is immediate, while \eqref{eq: induction annuli estimte} follows from \eqref{eq: annuli estimate 1}. Suppose now that \eqref{eq: induction annuli estimte} and \eqref{eq: induction estimte} hold for some integer $0\leq j<N$. By \eqref{ball-annulus-decomposition},
\begin{equation}
\label{eq: induction estimate chain}
\begin{split}
\mu(B(x,r_{j+1})\cap\Omega)&=
\mu(B(x,r_j)\cap\Omega)+\mu(A_{r_j,2r_j,\Omega})
\\
&\leq\mu(B(x,r_j)\cap\Omega)+\mu(A_{r_j,4r_j,\Omega})
\\
&\leq(1+\tau)\mu(B(x,r_j)\cap\Omega)
\\
&\leq(1+\tau)^{j+1}\mu(B(x,r)\cap\Omega).
\end{split}
\end{equation}
Thus, \eqref{eq: induction estimte} holds for $j+1$.

To prove that \eqref{eq: induction annuli estimte} holds for $j+1$, we apply \eqref{uslemma} using the radius $r_{j+1}<R_0<\widetilde{r}_{x,\tau}$ to obtain
\begin{equation}\label{eq: annuli estimate 2}
\min\{\mu(A_{r_{j+1},4r_{j+1},\Omega}),\mu(B(x,r_{j})\cap\Omega)\}\leq \tau\mu(B(x,r_{j+1})\cap\Omega),
\end{equation}
where we have also used the fact that $r_{j+1}/2=r_{j}$.
We claim that this minimum is equal to $\mu(A_{r_{j+1},4r_{j+1},\Omega})$.
Indeed, if not, then by \eqref{eq: induction estimate chain} and the fact that $\tau(1+\tau)<1$, we have
$$
\mu(B(x,r_j)\cap\Omega)\leq\tau\mu(B(x,r_{j+1})\cap\Omega)
\leq\tau(1+\tau)\mu(B(x,r_j)\cap\Omega)<
\mu(B(x,r_j)\cap\Omega),
$$
which is not possible since $\mu(B(x,r_j)\cap\Omega)\geq\mu(B(x,r)\cap\Omega)>0.$ 
Hence, the minimum in \eqref{eq: annuli estimate 2} equals 
$\mu(A_{r_{j+1},4r_{j+1},\Omega})$ and thus, \eqref{eq: induction annuli estimte} holds for $j+1$. The proof of \eqref{eq: induction annuli estimte} and \eqref{eq: induction estimte} is now complete.

Next, observe that if $y\in A_{R_0,2R_0,\Omega}$, then
$r_N<R_0\leq d(x,y)<2R_0\leq 4r_N$. Hence,
\begin{equation*}
\label{eq: annuli inclusion}
A_{R_0,2R_0,\Omega}\subset A_{r_N,4r_N,\Omega}.
\end{equation*}
Now, combining this with \eqref{eq: choice of sigma}, \eqref{eq: induction annuli estimte}, and \eqref{eq: induction estimte} yields
\begin{equation}
\begin{split}
\label{eq: sigma estimate}
0<\sigma&=\mu(A_{R_0,2R_0,\Omega})
\leq\mu(A_{r_N,4r_N,\Omega})
\\
&\leq\tau\mu(B(x,r_N)\cap\Omega)
\leq\tau(1+\tau)^N\mu(B(x,r)\cap\Omega).
\end{split}
\end{equation}
On the other hand, since $r_N=2^Nr$ and
$r_N<R_0<r_{\mathrm{rd}}$, the local reverse-doubling condition \eqref{local-reverse-doubling at x} gives
\begin{equation*}\label{reverse-doubling-application}
\begin{aligned}
\mu(B(x,r)\cap\Omega)
&\leq
\mu(B(x,r))\\
&\leq
C_{\mathrm{rd}}2^{-N\alpha}\mu(B(x,r_N))\\
&\leq
C_{\mathrm{rd}}2^{-N\alpha}\mu(B(x,R_0)).
\end{aligned}
\end{equation*}
Combining this with \eqref{eq: sigma estimate} yields
\begin{equation}
\label{eq: last contradiction}
0<\sigma\leq\tau C_{\mathrm{rd}}\mu(B(x,R_0))\left(\frac{1+\tau}{2^\alpha}\right)^N,
\end{equation}
where $\frac{1+\tau}{2^\alpha}<1$. Since the integer $N=N(r)$ defined by \eqref{choice-of-N} tends to infinity along any sequence of radii r satisfying \eqref{eq: contra} and tending to zero, the right-hand side of \eqref{eq: last contradiction} tends to zero, yielding a contradiction. Therefore, \eqref{eq: contra} cannot hold for
arbitrarily small radii and hence there exists $r_{x,c}>0$ such that the desired estimate \eqref{upper-decay-conclusion} holds
whenever $0<r<r_{x,c}.$
\end{proof}

\begin{lemma}\label{boundary-reverse-doubling}
Let $(X,d,\mu)$ be a doubling metric measure space, and let
$\Omega\subset X$ be a nonempty open connected set. Then $\mu$ is
locally reverse doubling at every point in $\partial\Omega$. More precisely,
there exist constants $C_{\mathrm{rd}}\geq 1$ and $\alpha>0$, depending
only on the doubling constant of $\mu$, such that for every
$x\in\partial\Omega$, there exists $r_x>0$ for which
\begin{equation*}\label{eq:boundary-local-reverse-doubling}
    \mu(B(x,r))
    \leq
    C_{\mathrm{rd}}
    \left(\frac{r}{R}\right)^\alpha
    \mu(B(x,R))
\end{equation*}
whenever $0<r\leq R<r_x$.
\end{lemma}

\begin{proof}
Fix $x\in\partial\Omega$. Since $\Omega$ is open, $x\notin\Omega$.
Choose $y_0\in\Omega$ and set $r_x:=d(x,y_0)>0$.
The function $y\mapsto d(x,y)$
is continuous on the connected set $\Omega$, and hence its image is an interval containing $r_x$. Moreover, since $x\in\partial\Omega$,
we have
$$
\inf_{y\in\Omega}d(x,y)=0,
$$
and hence the image of this map contains the interval $(0,r_x]$. As such, for every
$0<R<r_x$, there exists $z_R\in\Omega$ such that
$d(x,z_R)=3R/4$.
Observe that
\begin{equation}\label{eq:ball-in-boundary-annulus}
    B\left(z_R,R/8\right)
    \subset
    B(x,R)\setminus B(x,R/2).
\end{equation}

Let $C_D>1$ denote the doubling constant of $\mu$. Since
$B(x,R)\subset B(z_R,2R)$,
the doubling property gives
\[
    \mu(B(x,R))
    \leq
    \mu(B(z_R,2R))
    \leq
    C_D^4
    \mu\left(B\left(z_R,R/8\right)\right).
\]
Hence,
\[
    \mu\left(B\left(z_R,R/8\right)\right)
    \geq
    C_D^{-4}\mu(B(x,R)).
\]
Using \eqref{eq:ball-in-boundary-annulus}, we conclude that
\[
\begin{aligned}
    \mu(B(x,R/2))
    &\leq
    \mu(B(x,R))
    -
    \mu\left(B\left(z_R,R/8\right)\right)\\
    &\leq
    \left(1-C_D^{-4}\right)\mu(B(x,R)).
\end{aligned}
\]
Thus, with $\theta:=1-C_D^{-4}\in(0,1),$
we have
\begin{equation}\label{eq:one-step-boundary-reverse-doubling}
    \mu(B(x,R/2))
    \leq
    \theta\mu(B(x,R))
\end{equation}
for every $0<R<r_x$.
Iterating \eqref{eq:one-step-boundary-reverse-doubling} gives
\begin{equation}\label{eq:dyadic-reverse-doubling}
    \mu(B(x,2^{-j}R))
    \leq
    \theta^j\mu(B(x,R))
\end{equation}
for every $j\in\mathbb{N}_0$ and every $0<R<r_x$.

Let $\alpha:=\log_2\left(\theta^{-1}\right)>0.$
Given $0<r\leq R<r_x$, choose $j\in\mathbb{N}_0$ such that
$$2^{-j-1}R<r\leq2^{-j}R.$$
Then, by \eqref{eq:dyadic-reverse-doubling},
\[
\begin{aligned}
    \mu(B(x,r))
    &\leq
    \mu(B(x,2^{-j}R))\\
    &\leq
    \theta^j\mu(B(x,R))\\
    &\leq
    \theta^{-1}
    \left(\frac{r}{R}\right)^\alpha
    \mu(B(x,R)).
\end{aligned}
\]
Thus, $\mu$ satisfies the local reverse-doubling condition
\eqref{local-reverse-doubling at x} at every $x\in\partial\Omega$, and the  proof is complete.
\end{proof}

\section{Proofs of the main theorems}
\label{proofs of main results}
\subsection{Proof of Theorem~\ref{thmE.6}}
\begin{proof}
If $q\geq Q$, choose $0<q_0<Q$ sufficiently close to $Q$ so that $p<(q_0)^*_Q$. Then, by Lemma~\ref{small q extension}, 
$\Omega$ has the $(M^{1,p},M^{1,q_0})$-extension property. Thus, after replacing $q$ by $q_0$, we may assume $q<Q$.
Suppose, to the contrary, that $\mu(\partial \Omega)>0.$ 
Since $\mu$ is doubling and Borel regular, the Lebesgue differentiation theorem (see, e.g., \cite[Theorem~3.14]{AM15}) implies
that there exists a measurable set $E\subset\partial\Omega$ such that
$\mu(E)>0$  and
\begin{equation}\label{lbp}
    \lim_{r\to 0}\frac{\mu(\partial \Omega\cap B(x,r))}{\mu(B(x,r))}=1
\end{equation}
for all $x\in E.$
Since $\Omega$ is open, $\Omega\cap\partial\Omega=\varnothing$, and therefore
$$
B(x,r)\cap\Omega
\subset
B(x,r)\setminus\partial\Omega,
$$
from which we can deduce that
$$
0\leq\frac{\mu(B(x,r)\cap\Omega)}{\mu(B(x,r))}\leq
1-\frac{\mu(B(x,r)\cap\partial\Omega)}{\mu(B(x,r))},
$$
for each $r>0.$ Combining this with \eqref{lbp} gives
\begin{equation}\label{eq:Omega-density-zero-main}
    \lim_{r\to0^+}
    \frac{\mu(B(x,r)\cap\Omega)}
         {\mu(B(x,r))}
    =0
\end{equation}
for every $x\in E$.
On the other hand, by Lemma~\ref{lemmafun}, we have
$\overline D\Phi(x)<\infty$ for almost every $x\in X$. Since $\mu(E)>0$, we may choose
$x\in E$ such that
\begin{equation}\label{case 2(1)}
    \overline{D\Phi}(x)<\infty.
\end{equation} 

Now, let $0<c<2^{-s}$, 
where $s\in(0,\infty)$ is any exponent satisfying the conclusion of Lemma~\ref{lem: lower measure bound}.
In view of \eqref{eq:Omega-density-zero-main}, \eqref{case 2(1)}, and Lemma~\ref{boundary-reverse-doubling}, we can appeal to Lemma~\ref{upper}, to conclude that there exists $r_{x,c}>0$ such that
\begin{equation}\label{4.3}
    \mu(B(x,r/2)\cap \Omega)\leq c\mu(B(x,r)\cap\Omega)
\end{equation}
for all $0<r<r_{x,c}$. 
Fix $0<r<\min\{\diam(\Omega),r_{x,c}\}$.
By iterating \eqref{4.3}, we conclude that for $j\in \mathbb{N}$
\begin{equation}\label{Literation}
    \mu\left(B\left(x,r/2^{j}\right)\cap\Omega\right)\leq c^j\mu(B(x,r)\cap\Omega).
\end{equation}
On the other hand, by Lemma~\ref{lem: lower measure bound}, we also have
\begin{equation}\label{Uiteration}
    \mu\left(B\left(x,r/2^j\right)\cap\Omega\right)\geq \theta (r/2^{j})^s,
\end{equation}
for all $j\in \mathbb{N}$.
From \eqref{Uiteration} and \eqref{Literation}, we conclude that
\begin{equation*}
    c^j\mu(B(x,r)\cap\Omega)\geq \theta (r/2^{j})^s,
\end{equation*}
equivalently,
\begin{equation}\label{equa}
\theta r^s\leq (2^sc)^j\mu(B(x,r)\cap\Omega).
\end{equation}
Since $2^sc<1$ and $\mu(B(x,r)\cap\Omega)<\infty$, letting $j\to\infty$ in \eqref{equa} gives $\theta r^s\leq0$,
which cannot hold. Therefore, we must have $\mu(\partial \Omega)=0$, and the proof of Theorem~\ref{thmE.6} is complete.
\end{proof}

\subsection{Proof of Theorem~\ref{thmE.6-notconnected}}

\begin{proof}
Since $\mu$ is $Q$-Ahlfors regular, $\mu$ is $Q$-doubling and satisfies the local reverse-doubling condition
\eqref{local-reverse-doubling at x} at every $x\in X$, with
$\alpha=Q$. In light of this, the remainder of the proof is identical to the proof of Theorem~\ref{thmE.6}. In particular, connectedness of $\Omega$ is no longer necessary, since it was used only to obtain the local reverse-doubling condition at boundary points; see Lemma~\ref{boundary-reverse-doubling}.
\end{proof}

\begin{remark}\label{RDspaces}
{\rm The conclusion of Theorem~\ref{thmE.6-notconnected} is still valid if  $Q$-Ahlfors regularity  is replaced by the weaker assumption that $\mu$ is $Q$-doubling and satisfies the local reverse-doubling condition \eqref{local-reverse-doubling at x} at almost every point in $\partial\Omega$, where the constants $C_{\mathrm{rd}}$, $r_{\mathrm{rd}}$, and $\alpha$ are allowed to depend on $x$. This follows from the same proof, since the local reverse-doubling condition is needed only at the boundary point selected in the argument. In particular, the conclusion of Theorem~\ref{thmE.6-notconnected} holds in RD-spaces; see, e.g., \cite[Section~1.1]{RDspaces}.}
\end{remark}





\end{document}